\documentclass[12pt]{amsart}
\usepackage{amssymb}
\usepackage{amsfonts}
\usepackage{amssymb,latexsym}
\usepackage{enumerate}
\usepackage{url}
\usepackage{mathrsfs}
\usepackage{bbm}

\usepackage{amsmath,amscd}
\usepackage[T1]{fontenc}

\usepackage{tikz}
\usetikzlibrary{arrows}
\tikzstyle{block}=[draw opacity=0.7,line width=1.4cm]

\makeatletter
\@namedef{subjclassname@2010}{%
  \textup{2010} Mathematics Subject Classification}
\makeatother

  [2002/01/19 v2.2g %
    AMS font definitions%
  ]
\DeclareFontFamily{U}{euf}{}
\DeclareFontShape{U}{euf}{m}{n}{%
  <5><6><7><8><9>gen*eufm%
  <10><10.95><12><14.4><17.28><20.74><24.88>eufm10%
  }{}
\DeclareFontShape{U}{euf}{b}{n}{%
  <5><6><7><8><9>gen*eufb%
  <10><10.95><12><14.4><17.28><20.74><24.88>eufb10%
  }{}

  [2002/01/19 v2.2g %
    AMS font definitions%
  ]
\DeclareFontFamily{U}{msb}{}
\DeclareFontShape{U}{msb}{m}{n}{%
  <5><6><7><8><9>gen*msbm%
  <10><10.95><12><14.4><17.28><20.74><24.88>msbm10%
  }{}

  [2002/01/19 v2.2g %
    AMS font definitions%
  ]
\DeclareFontFamily{U}{msa}{}
\DeclareFontShape{U}{msa}{m}{n}{%
  <5><6><7><8><9>gen*msam%
  <10><10.95><12><14.4><17.28><20.74><24.88>msam10%
  }{}

\newtheorem{theorem}{Theorem}[section]

\newtheorem{proposition}[theorem]{Proposition}

\theoremstyle{definition}

\theoremstyle{remark}
\newtheorem{remark}[theorem]{Remark}

\numberwithin{equation}{section} 

\begin{document}

\title[Zeros of a general family of zeta functions]{Hamiltonians for the zeros of a general family of zeta functions}

\author{Su Hu}
\address{Department of Mathematics, South China University of Technology, Guangzhou, Guangdong 510640, China}
\email{mahusu@scut.edu.cn}

\author{Min-Soo Kim}
\address{Department of Mathematics Education, Kyungnam University, Changwon, Gyeongnam 51767, Republic of Korea}
\email{mskim@kyungnam.ac.kr}


\begin{abstract}
Towards the  Hilbert-P\'olya conjecture, in this paper, we present a general construction of Hamiltonian $\hat{H}_{f},$ which leads a general family of Hurwitz zeta functions $(-1)^{z_{n}-1}L(f,z_{n},x+1)$ defined by Mellin transform becomes their eigenstates under a suitable boundary condition, and the eigenvalues $E_{n}$ have the property that $z_{n}=\frac{1}{2}(1-iE_{n})$ are the zeros of a general family of zeta functions $L(f,z)\equiv L(f,z,1)$.
\end{abstract}

\maketitle
\def\ord{\text{ord}_p}
\def\ordt{\text{ord}_2}
\def\o{\omega}
\def\la{\langle}
\def\ra{\rangle}
\def\Log{{\rm Log}\, \Gamma_{\! D,E,N}}

\section{Introduction}
The Riemann zeta function $\zeta(z)$ is conventionally represented as the sum or the integral
\begin{equation}\label{(1)}
\zeta(z)=\sum_{n=1}^{\infty}\frac{1}{n^{z}}=\frac{1}{\Gamma(z)}\int_{0}^{\infty}\frac{t^{z-1}}{e^{t}-1}dt,\quad\textrm{Re}(z) > 1
\end{equation}
(\cite[Theorem 12.2]{Apostol}). The integral reduces to the sum if the denominator of the integrand is expanded in  geometric series.
It is well-known that $\zeta(z)$ can be analytically continued to the whole complex plane except for a single pole at $z=1$ with residue 1
and $\zeta(z)$ satisfies the two equivalent functional equations (we present  one of them here, the other can be obtained from this by letting $z\to1-z$)
\begin{equation}\label{Functional}
\zeta(z)=2(2\pi)^{1-z}\Gamma(1-z)\sin\left(\frac{\pi z}{2}\right)\zeta(1-z),\quad\textrm{Re}(z) < 1
\end{equation}
(\cite[Theorem 12.7]{Apostol}).
From the above equation, we see that $\zeta(z)$ is symmetric with the axis $\textrm{Re}(z)=\frac{1}{2}$.

By (\ref{Functional}), from the vanishing of the function $\sin\left(\frac{\pi z}{2}\right)$, the negative even integers are also zeros of $\zeta(z)$,
which are called the trivial zeros.
Furthermore, from the representation of the zeta function in terms of products involving prime numbers, it can be shown that the zeta function has no zero for Re$(z)>1$ since none of the factors can vanish there.
In 1859, Riemann \cite{Riemann} conjectured that the nontrivial zeros of $\zeta(z)$ are all lie on the critical line $\textrm{Re}(z)=\frac{1}{2},$
namely,
\begin{equation}\label{Re-zero}
\zeta\left(\frac12+ie_* \right)=0,
\end{equation}
where $e_*$ denotes the location of a zero on the critical line.
This becomes one of most remarkable conjectures in mathematics, known as the Riemann hypothesis.
To the purpose of solving Riemann hypothesis, Hilbert-P\'olya \cite{Hilbert} conjectured that the imaginary parts of the zeros of $\zeta(z)$
might correspond  to the eigenvalues of a Hermitian, self-adjoint operator.
Berry and Keating \cite{Berry} conjectured that the classical counterpart of such a Hamiltonian may have the form $\hat{H}=\hat{x}\hat{p}$,
but a Hamiltonian possessing  this property has not yet been found. In quantum mechanics, associated with each measurable parameter in a physical system is a quantum mechanical operator, and the operator associated with the system energy is called the Hamiltonian.

In 2016, Bender, Brody and M\"uller  \cite{Bender} found an interesting correspondence between the nontrivial zeros of
$\zeta(z)$ and the eigenvalues of Hamiltonian for a quantum system.
That is, let $\hat x$ be the position operator and let $\hat{p}=-i\frac d{dx}$ be the momentum operator. 
They constructed a non-Hermitian Hamiltonian
\begin{equation}
\hat{H}=\frac{\mathbbm{1}}{\mathbbm{1}-e^{-i\hat{p}}}(\hat{x}\hat{p}+\hat{p}\hat{x})({\mathbbm{1}-e^{-i\hat{p}}}),
\end{equation}
which satisfies the conditions of the Hilbert-P\'olya conjecture, that is, if the eigenfunctions of $\hat{H}$ satisfy the boundary condition $\psi_{n}(0)=0$ for all $n$, then
the real eigenvalues $E_{n}$ have the property that $\frac{1}{2}(1-iE_{n})$ are the nontrivial zeros of the Riemann zeta function, and the Hurwitz zeta functions
$\psi_{n}(x)=-\zeta(z_{n},x+1)$ are the corresponding eigenstates. They also constructed the metric operator
to define an inner-product space, on which the Hamiltonian is Hermitian. They remarked that if the analysis may be made rigorous to show that
$\hat{H}$ is manifestly self-adjoint, then this implies that the Riemann hypothesis holds true.

In this paper, we present a general construction of Hamiltonian $\hat{H}_{f},$ which leads   a general family of Hurwitz zeta functions $(-1)^{z_{n}-1}L(f,z_{n},x+1)$ becomes their eigenstates under a suitable boundary condition, 
and the eigenvalues $E_{n}$ have the property that $z_{n}=\frac{1}{2}(1-iE_{n})$ are the  zeros of a general family of zeta functions $L(f,z)$ (see Theorem \ref{main} below).
The definitions of $L(f,z,x)$ and  $L(f,z)$ will be given in the next section.

\section{Mellin transform and zeta functions}
Recall that a function $f$ tends rapidly to $0$ at infinity if for any $k\geq 0$, as $t\to +\infty$ the functions $t^{k}f(t)$ tends to $0$. Let  $tf(-t)$ be a $C^{\infty}$ function on $[0,\infty)$
tending rapidly to $0$ at infinity and for $\textrm{Re}(z)>1$, we define a general family of zeta functions $L(f,z)$ from the following integral representation
 \begin{equation}\label{(5)}
L(f,z)=\frac{1}{\Gamma(z)}\int_{0}^{\infty}t^{z}f(-t)\frac{dt}{t}.
\end{equation}
As pointed out by Cohen \cite[Proposition 10.2.2(2)]{Cohen}, if $f(-t)$ itself is a $C^{\infty}$ function on $[0,\infty)$ and
it tends rapidly to $0$ at infinity, then we have $L(f,z)$ can be analytically  continued to the whole of $\mathbb{C}$ into a holomorphic function.

The  Huwitz zeta function $\zeta(z,x)$ is defined by the series
\begin{equation}
\zeta(z,x)=\sum_{n=0}^{\infty}\frac{1}{(n+x)^{z}},\quad\textrm{Re}(z)>1,
\end{equation}
where $x\in (0,\infty)$  (see Cohen \cite[p.~72, Definition 9.6.1]{Cohen} and the discussions for the definition area of $x$ there).
It has the integral representation
\begin{equation}
\zeta(z,x)=\frac{1}{\Gamma(z)}\int_{0}^{\infty}e^{-(x-1)t}\frac{t^{z-1}}{e^{t}-1}dt.
\end{equation}

Inspired from this, we also define a general family of Hurwitz zeta function $L(f, z,x)$ as the following integral representation
\begin{equation}\label{(6)}
L(f,z,x)=\frac{1}{\Gamma(z)}\int_{0}^{\infty} t^{z}e^{-(x-1)t} f(-t)\frac{dt}{t},
\end{equation}
where Re$(z)>1$ and $x\in (0,\infty).$
It is easy to see that $L(f,z,1)=L(f,z)$.

By setting $f(t)$ as different functions in (\ref{(5)}), $L(f,z,x)$ and $L(f,z)$ become
the well-known Riemann zeta functions, Hurwitz zeta functions, Dirichlet $L$-functions, Hecke series and other zeta functions. 
The following are some examples. 
\begin{itemize}
\item
Letting $f(t)=\frac{1}{e^{-t}-1}$ in (\ref{(5)}) and (\ref{(6)}), we have the Riemann zeta function $\zeta(z)$ and the Hurwitz zeta function $\zeta(z,x)$, respectively (see Colmez's Tsinghua lecture notes \cite[p.~4]{Colmez}).

\item 
Suppose  $\chi$ is a Dirichlet character modulo $m$ and $$L(\chi,z)=\sum_{n=1}^{\infty}\frac{\chi(n)}{n^{z}}=\prod_{p}\left(1-\frac{\chi{(p)}}{p^{z}}\right)^{-1} $$
is the Dirichlet $L$-function. From the integral representation of the Gamma function
$$\int_{0}^{\infty}e^{-nt}t^{z}\frac{dt}{t}=\frac{\Gamma(z)}{n^{z}},$$
we have $$L(\chi,z)=\frac{1}{\Gamma(z)} \int_{0}^{\infty}G_{\chi}(t)t^{z-1}dt,\quad\textrm{Re}(z)>1,$$
where 
$$
\begin{aligned}
G_{\chi}(t)&=\sum_{n=1}^{\infty}\chi(n)e^{-nt}=\sum_{r=1}^{m}\chi(r)\sum_{q=0}^{\infty}e^{-(r+qm)t}
\\
&=\sum_{r=1}^m\frac{\chi(r)e^{-rt}}{1-e^{-mt}}
\end{aligned}
$$
(see \cite[p.~160]{Cohen} and \cite[p.~102]{Iwa}).
Thus letting $f(t)=G_{\chi}(-t)$ in (\ref{(5)}) and (\ref{(6)}), 
we recover the Dirichlet $L$-function $L(\chi,z)$ (see \cite{Iwa}) and the two variable Dirichlet $L$-function $L(\chi,z,x)$
(see \cite{fox}).

\item 
Let  
$$
\lambda(z)=\sum_{n=0}^\infty\frac1{(2n+1)^z},\quad\text{Re}(z)>1
$$
be the Dirichlet lambda function according to Abramowitz and Stegun's handbook \cite[pp.~807--808]{AS}, 
which has also been studied by Euler under the notation $N(z)$ (see \cite[p.~70]{Varadarajan}).
By \cite[(2.10)]{HK}, we have the following integral representation
$$
\lambda(z)=\frac1{\Gamma(z)}\int_0^\infty\frac{e^{t}t^{z-1}}{e^{2t}-1}dt, \quad\text{Re}(z)>1.
$$
Thus letting $f(t)=\frac{e^{-t}}{e^{-2t}-1}$ in (\ref{(5)}), we recover the Dirichlet lambda function $\lambda(z)$.

\item 
Let $w$ be an even integer, and let $S_{w+2}$ be the space of cusp forms with respect to $\Gamma=SL(2,\mathbb Z)/(\pm1)$ 
on the upper half-plane of one complex variable. For $\Phi\in S_{w+2},$ 
we write $\Phi$ as a Fourier series
$\Phi(z)=\sum_{n=1}^\infty\lambda_n e^{2\pi inz}.$
Let $L_\Phi(z)$ be the Hecke series of $\Phi.$ Namely,
$$L_\Phi(z)=\sum_{n=1}^\infty\frac{\lambda_n}{n^z},$$
under the condition: there exists $z_{0} > 0$, such that Re$(z)\geq z_{0}$.

It is clear that $L_\Phi(s)$ has the following expression:
$$\begin{aligned}
L_\Phi(z)
&=\frac{(2\pi)^z}{\Gamma(z)}\sum_{n=1}^\infty{\lambda_n}\int_0^\infty e^{-2\pi nt}t^{z-1}dt \\
&=\frac{(2\pi)^z}{\Gamma(z)}\int_0^\infty \Phi(it)t^{z-1}dt.
\end{aligned}$$
So letting $f(-t)=(2\pi)^z\Phi(it)$ in (\ref{(5)}), we recover the Hecke series $L_\Phi(z)$
(see \cite{Man}).
\end{itemize}

Assume the function $tf(t)$ is analytic on an open disc  of radius $r$ centered at 0 in the complex plane $\mathbb{C}$, we have the following result.

\begin{proposition}\label{2.1} 
Let $C$ be a loop around the negative real axis. If  $x\in (0,\infty)$ the function defined by the contour integral
$$
I(f,z,x)=\frac{1}{2\pi i}\int_{C}  t^{z}e^{(1-x)t} f(t)\frac{dt}{t}
$$
is an entire function of $z$.
Moreover, we have
$$
L(f,z,x)=\Gamma(1-z)I(f,z,x), \quad{\rm Re}(z)>1.
$$
\end{proposition}

\begin{remark}\label{rem-2} 
If $\textrm{Re}(z)\leq 1$ and $x\in (0,\infty),$ we define $L(f,z,x)$ by the equation
$$
L(f,z,x)=\Gamma(1-z)I(f,z,x).
$$
This equation provides the analytic continuation of $L(f,z,x)$ in the entire complex plane. 
\end{remark}

\begin{proof} The proof follows from similar arguments as \cite[Theorem 12.3]{Apostol}.

We regard the interval $[0,\infty)$ of integration as a path of a complex integral, and then expanding this a little. Here we consider the following
contour $C.$
For $\varepsilon>0,$ we define $C$ by a curve $\varphi:(-\infty,\infty)\to\mathbb C$ given by
\begin{equation}
C  : \quad \varphi(u)=\begin{cases}
u &\text{if }u<-\varepsilon, \\
\varepsilon\exp\left(\pi i \frac{u+\varepsilon}{\varepsilon}\right) &\text{if }-\varepsilon\leq u\leq \varepsilon, \\
-u &\text{if }u>\varepsilon.
\end{cases}
\end{equation}
In the definition of $C,$ the parts for $u<-\varepsilon$ and $u>\varepsilon$ overlap, but we interpret
it that for $u<-\varepsilon$ we take the path above the real axis and for $u>\varepsilon$ below the real axis. This path is illustrated in Figure 1.

\begin{figure}[h]
\centerline{ \includegraphics[width=3.5in, height=0.77in]{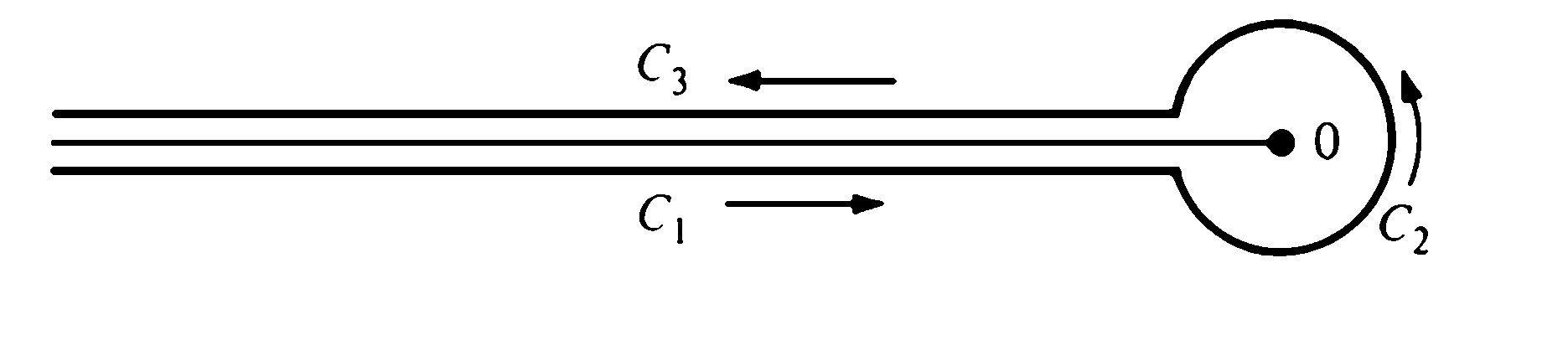}}
\caption{Path of $C$}
\end{figure}

Now consider the complex contour integral
\begin{equation}\label{c-int}
\int_{C}t^{z}e^{(1-x)t}f(t)\frac{dt}{t}.
\end{equation}
Since we have to treat $t^{z-1}$ on $C,$ we shall choose a complex power $t^z$ for $z\in\mathbb C.$
Denoting the argument of $t$ by $\arg t,$ we can define a single-valued function $\log t$ on
$\mathbb C-\{s=a+ib\mid b=0,a\leq0\}$ by
\begin{equation}
\log t=\log|t|+i\arg t \quad(-\pi<\arg t<\pi).
\end{equation}
Using this, a single-valued analytic function $t^z$ on $\mathbb C-\{s=a+ib\mid b=0,a\leq0\}$
is defined by
\begin{equation}
t^z=e^{z\log t}=e^{z(\log|t|+i\arg t)},
\end{equation}
where $-\pi<\arg t<\pi.$
We divide the contour $C$ into three pieces as follows.

$\quad C_1$: the part of the real axis from $-\infty$ to $-\varepsilon,$

$\quad C_2$: the positively oriented circle of radius $\varepsilon$ with center at the origin,

$\quad C_3$: the part of the real axis from $-\varepsilon$ to $-\infty.$

\noindent
We have $C=C_1+C_2+C_3.$ For $t$ on $C_1,$ we put $\arg t=-\pi,$ and for $t$ on $C_{3},$ we put $\arg t=\pi.$
Then the integral on $C_1$ is given by
\begin{equation}
\begin{aligned}
\int_{C_1}t^{z}e^{(1-x)t}f(t)\frac{dt}{t}&=\int_{\infty}^\varepsilon t^{z-1}e^{-\pi iz}e^{-(1-x)t}f(-t)dt \\
&=-e^{-\pi iz}\int^{\infty}_\varepsilon t^{z-1}e^{-(1-x)t}f(-t)dt
\end{aligned}
\end{equation}
and on $C_3$ by
\begin{equation}
\begin{aligned}
\int_{C_3}t^{z}e^{(1-x)t}f(t)\frac{dt}{t}&=\int_\varepsilon^{\infty} t^{z-1}e^{\pi iz}e^{-(1-x)t}f(-t)dt \\
&=e^{\pi iz}\int^{\infty}_\varepsilon t^{z-1}e^{-(1-x)t}f(-t)dt.
\end{aligned}
\end{equation}
Thus, together we get
\begin{equation}\label{inte 1-3}
\begin{aligned}
\int_{C}t^{z}e^{(1-x)t}f(t)\frac{dt}{t}&=(e^{\pi iz}-e^{-\pi iz})\int^{\infty}_\varepsilon t^{z-1}e^{-(1-x)t}f(-t)dt \\
&+\int_{C_2}t^{z}e^{(1-x)t}f(t)\frac{dt}{t}.
\end{aligned}
\end{equation}
The circle $C_2$ is parameterized as $t=\varepsilon e^{i\theta}(-\pi\leq \theta\leq\pi),$ so on $C_2$ we have
$t^{z-1}=\varepsilon^{z-1}e^{i(z-1)\theta},$
and the absolute value of the integral is estimated from above as
\begin{equation}\label{inte-esi}
\begin{aligned}
\left|\int_{C_2}t^{z}e^{(1-x)t}f(t)\frac{dt}{t}\right|
&\leq \int_{-\pi}^\pi \left| \varepsilon^{z-1} e^{i(z-1)\theta}e^{(1-x)\varepsilon e^{i\theta}}f(\varepsilon e^{i\theta})i\varepsilon
e^{i\theta}\right| d\theta \\
&=\varepsilon^{\text{Re}(z)}
 \int_{-\pi}^\pi \left| e^{iz\theta}e^{(1-x)\varepsilon e^{i\theta}} f(\varepsilon e^{i\theta})\right| d\theta.
\end{aligned}
\end{equation}
Since  $tf(t)$ is analytic on an open disc  of radius $r$ centered at 0, we have $\varepsilon f(\varepsilon e^{i\theta})$ is bounded as a function of $\varepsilon$ and $\theta$
if $\varepsilon\leq r/2$  and $-\pi\leq \theta\leq\pi$.
So if we take the limit $\varepsilon\to0$, then, for Re$(z)>1,$ by (\ref{inte-esi}) we get
\begin{equation}\label{inte-esi-0}
\lim_{\varepsilon\to0}\int_{C_2}t^{z}e^{(1-x)t}f(t)\frac{dt}{t}=0.
\end{equation}
Therefore, using (\ref{inte-esi-0}), we get
\begin{equation}\label{inte-final}
\int_{C}t^{z}e^{(1-x)t}f(t)\frac{dt}{t}=(e^{\pi iz}-e^{-\pi iz})\int_0^{\infty} t^{z-1}e^{-(1-x)t}f(-t)dt.
\end{equation}
Let
\begin{equation}
I(f,z,x)=\frac{1}{2\pi i}\int_{C}  t^{z}e^{(1-x)t} f(t)\frac{dt}{t}.
\end{equation}
By (\ref{(6)}) and (\ref{inte-final}), we have
\begin{equation}\label{inte-final-1}
\begin{aligned}
I(f,z,x)&=\frac{1}{2\pi i}\int_{C}  t^{z}e^{(1-x)t} f(t)\frac{dt}{t} \\
&=\frac{e^{\pi iz}-e^{-\pi iz}}{2\pi i}\int_0^{\infty} t^{z-1}e^{-(1-x)t}f(-t)dt \\
&=\frac{\sin\pi z}{\pi}\Gamma(z)L(f,z,x) \\
&=\frac1{\Gamma(1-z)}L(f,z,x),
\end{aligned}
\end{equation}
which is the desired result.
\end{proof}

\section{Main results}

Denote by $\mathbb{R}^{+}=[0,\infty)$ and $H=L^{2}(\mathbb{R}^{+})$ be the Hilbert space of all complex valued functions $g,$
defined almost everywhere on $\mathbb{R^{+}}$ and $|g|^{2}$ is Lebesgue integrable, with the pointwise operations
of additions and the multiplication-by-scalars, and with the $L^{2}-$norm
\begin{equation}
\|g\|_{2}=\left(\int_{\mathbb{R}^{+}}|g(x)|^{2}dx\right)^{1/2}.
\end{equation}
Give a function $f(t)$ such that $tf(-t)$ be a $C^{\infty}$ function on $[0,\infty)$
tending rapidly to $0$ at infinity. Assume $tf(t)$ is analytic on an open disc  of radius $r$ centered at 0 in the complex plane $\mathbb{C}$ and it has the following power series expansion at $t=0$
\begin{equation}\label{(3.1)}
(-t)f(t)=\sum_{n=0}^{\infty}a_{n}t^{n}
\end{equation}
with $a_{0}= 0$ and $a_{1}\neq 0.$ Define an operator on $H$ by
\begin{equation}\label{delta}
\hat{\Delta}_{f}\Psi(x)=\frac{\mathbbm{1}}{f(-i\hat{p})}\Psi(x),
\end{equation}
for any $\Psi(x)\in H$. Here we recall that $\hat{p}=-i\frac d{dx}$ denotes the momentum operator. Then by (\ref{delta}), 
\begin{equation}
\hat{\Delta}_{f}^{-1}\Psi(x)=f(-i\hat{p})\Psi(x),
\end{equation}
and from (\ref{(3.1)}), we get
\begin{equation}\label{op-f}
\hat{\Delta}_{f}^{-1}\Psi(x)=f(-i\hat{p})\Psi(x)=\frac{\mathbbm{1}}{i\hat{p}}\sum_{n=0}^{\infty}a_{n}(-i\hat{p})^{n}\Psi(x).
\end{equation}
Unfortunately, the above series do not always converge. However, if we truncate it by setting
\begin{equation}\label{delta-inv-tr}
s_N(x)=\frac{\mathbbm{1}}{i\hat p}\sum_{n=0}^Na_{n}(-i\hat p)^{n}\Psi(x)
\end{equation}
for some integer $N\geq1,$ then we require that
\begin{equation}\label{delta-inv-tr1}
\hat{\Delta}_{f}^{-1}\Psi(x)=s_N(x)+O(x^{-N}) \quad\text{as } x\to\infty.
\end{equation}

The operator $\hat{\Delta}_{f}$ has the following two examples.

First, letting $\{a_n\}=\{B_n\}$ in (\ref{(3.1)}), where $B_n$ are the Bernoulli numbers (see \cite[p. 803]{AS}). 
Then by (\ref{op-f}), we obtain
\begin{equation}\label{op-f-ex1}
\hat{\Delta}_{f}^{-1}=f(-i\hat p)=\frac{\mathbbm{1}}{i\hat p}\sum_{n=0}^{\infty}B_{n}\frac{(-i\hat p)^{n}}{n!}.
\end{equation}
From the generating function of Bernoulli numbers, we have
\begin{equation}\label{op-f-ex1-1}
\hat{\Delta}_{f}^{-1}=\frac{\mathbbm{1}}{i\hat p}\sum_{n=0}^{\infty}B_{n}\frac{(-i\hat p)^{n}}{n!}=\frac{\mathbbm{1}}{i\hat p}\frac{-i\hat p}{e^{-i\hat p}-\mathbbm{1}},
\end{equation}
which is equivalent to $\hat{\Delta}_{f}=\mathbbm{1}-e^{-i\hat p}$ (see \cite{Bender}).

Next, let $\{a_n\}=\{E_n(0)\}$ in (\ref{(3.1)}), where $E_n(x)$ are the Euler polynomials (see \cite[p. 803]{AS}). 
Similarly, by (\ref{op-f}), we have
\begin{equation}\label{op-f-ex2}
\hat{\Delta}_{f}^{-1}=f(-i\hat p)=\frac{\mathbbm{1}}{i\hat p}\sum_{n=0}^{\infty}E_{n}(0)\frac{(-i\hat p)^{n}}{n!}.
\end{equation}
From the generating function of Euler polynomials, we have
\begin{equation}\label{op-f-ex2-1}
\hat{\Delta}_{f}^{-1}=\frac{\mathbbm{1}}{i\hat p}\sum_{n=0}^{\infty}E_{n}(0)\frac{(-i\hat p)^{n}}{n!}=\frac{2}{i\hat p}\frac{\mathbbm{1}}{e^{-i\hat p}+\mathbbm{1}},
\end{equation}
which is equivalent to $\hat{\Delta}_{f}=\frac12(i\hat p)(\mathbbm{1}+e^{-i\hat p}).$

We are at the position to state the main result. 

\begin{theorem}\label{main}
Let $\hat{p}=-i\frac d{dx}$ be the momentum operator 
and $\hat{H}_{f}=\hat{\Delta}_{f}^{-1}(\hat{x}\hat{p}+\hat{p}\hat{x})\hat{\Delta}_{f}$ be an operator
on the Hilbert space $H=L^{2}(\mathbb{R}^{+}),$ where $\hat x$ is the position operator. If $z\neq 1$, then
$$
\hat{H}_{f}\Psi(f,z,x)=i(2z-1)\Psi(f,z,x),
$$
where $\Psi(f,z,x)=(-1)^{z-1}L(f,z,x+1)$ and $L(f,z,x)$ is defined in (\ref{(6)}) and Remark \ref{rem-2}.
Furthermore, if we propose the boundary condition
$$
\Psi(f,z,0)=0
$$
to the differential equation $\hat{H}_{f}\Psi(f,z,x)=i(2z-1)\Psi(f,z,x),$ then we have
the $n$th eigenstates of the Hamiltonian $\hat{H}_{f}$ are
$$\Psi(f,z_{n},x)=(-1)^{z_{n}-1}L(f,z_{n},x+1),$$ 
the $n$th eigenvalues are $E_{n}=i(2z_{n}-1)$ for all $n,$ and
from the boundary condition, $z_{n}=\frac{1}{2}(1-iE_{n})$ are the zeros of the general family zeta functions $L(f,z)\equiv L(f,z,1)$.
\end{theorem}
\begin{proof} 
By  (\ref{op-f}), 
\begin{equation}\label{op-f-a}
\hat{\Delta}_{f}^{-1}\Psi(x)=\frac{\mathbbm{1}}{i\hat{p}}\sum_{n=0}^{\infty}a_{n}(-i\hat{p})^{n}\Psi(x).
\end{equation}
Since $\hat{p}=-i\frac d{dx}$ and $i\hat{p}=\frac d{dx}$, for $z\neq 1$, we have
\begin{equation}\label{op-xz-1}
\begin{aligned}
i\hat{p}\left(\frac{x^{1-z}}{1-z}\right)&=i(-i)\frac d{dx}\left[\frac{x^{1-z}}{1-z}\right]\\
&=\frac d{dx}\left(\frac{x^{1-z}}{1-z}\right)\\
&=x^{-z}.
\end{aligned}
\end{equation}Thus
\begin{equation}\label{op-xz-1}
\frac{\mathbbm{1}}{i\hat{p}}x^{-z}=\frac{x^{1-z}}{1-z}.
\end{equation}
By (\ref{op-f-a}) with $\Psi(x)=x^{-z}$ and (\ref{op-xz-1}), we have
\begin{equation}\label{(7)}
\begin{aligned}
\hat{\Delta}_{f}^{-1}x^{-z}
&=\sum_{n=0}^{\infty}a_{n}(-i\hat{p})^{n}\frac{\mathbbm{1}}{i\hat{p}}x^{-z} \\
&=\sum_{n=0}^{\infty}a_{n}(-i\hat{p})^{n}\left(\frac{x^{1-z}}{1-z}\right)\\
&=\frac{1}{1-z}\sum_{n=0}^{\infty}a_{n}(-i\hat{p})^{n}x^{1-z}.
\end{aligned}
\end{equation}
Since $i\hat{p}=\frac d{dx}$ and $\left(\frac {d}{dx}\right)^nx^{\mu}=\frac{\Gamma(\mu+1)}{\Gamma(\mu-n+1)}x^{\mu-n}$, from (\ref{(7)}), setting $\mu=1-z,$ we obtain the asymptotic series (see \cite[(4)]{Bender})
\begin{equation}\label{(8)}
 \hat{\Delta}_{f}^{-1}x^{-z}
=\frac{\Gamma(2-z)}{1-z}\sum_{n=0}^{\infty}a_{n}(-1)^{n}\frac{x^{1-z-n}}{\Gamma(2-z-n)},
 \end{equation}
 which is valid under the limit $x\to\infty.$
 Since $\Gamma(2-z-n)$ has the following integral representation
 \begin{equation*}
 \frac{1}{\Gamma(2-z-n)}=\frac{1}{2\pi i}\int_{C}e^{u}u^{n+z-2}du,
 \end{equation*}
 where $C$ denotes a Hankel contour that encircles the negative-$u$ axis in the positive orientation (see \cite[p.~255]{AS}),
then by (\ref{(3.1)}) and (\ref{(8)}) we have
 \begin{equation}\label{(9)}
 \begin{aligned}
 \hat{\Delta}_{f}^{-1}x^{-z}
&=\frac{\Gamma(1-z)}{2\pi i}x^{1-z}\int_{C}e^{u}u^{z-2}du\sum_{n=0}^{\infty}a_{n}\left(-\frac{u}{x}\right)^{n}\\
&=\frac{\Gamma(1-z)}{2\pi i}x^{1-z}\int_{C}e^{u}u^{z-2}f\left(-\frac{u}{x}\right)\left(\frac{u}{x}\right)du.
\end{aligned}
 \end{equation}
 Letting $\frac{u}{x}=t$ in (\ref{(9)}), we have
 \begin{equation}
 \begin{aligned}
 \hat{\Delta}_{f}^{-1}x^{-z}&=\frac{\Gamma(1-z)}{2\pi i}\int_{C} t^{z}e^{xt} f(-t)\frac{dt}{t}\\
 &=(-1)^{z-1}\frac{\Gamma(1-z)}{2\pi i}\int_{C}t^{z}e^{-xt} f(t)\frac{dt}{t} \\
 &=(-1)^{z-1}\Gamma(1-z)\frac{1}{2\pi i}\int_{C}t^{z}e^{(1-(x+1))t} f(t)\frac{dt}{t}.
 \end{aligned}
 \end{equation}
Therefore, by Proposition \ref{2.1} and Remark \ref{rem-2},
\begin{equation}\label{(10)}
 \hat{\Delta}_{f}^{-1}x^{-z}=(-1)^{z-1}L(f,z,x+1).
 \end{equation}
 Note that
 \begin{equation}\label{note}
 \begin{aligned}
 (\hat x\hat p+\hat p\hat x)(x^{-z})&=-i\left(2x\frac d{dx}+\mathbbm{1}\right)(x^{-z}) \\
 &=-i\left(2x\frac d{dx}(x^{-z}) +\mathbbm{1}(x^{-z}) \right) \\
 &=-i\left(2x(-z)x^{-z-1}+x^{-z}\right) \\
 &=i(2z-1)x^{-z}.
 \end{aligned}
 \end{equation}
 Let $\hat{H}_{f}=\hat{\Delta}_{f}^{-1}(\hat{x}\hat{p}+\hat{p}\hat{x})\hat{\Delta}_{f}$.
 Denote by $\Psi(f,z,x)=(-1)^{z-1}L(f,z,x+1)$,
 from (\ref{(10)}) and (\ref{note}), we have $
 \Psi(f,z,x)=\hat{\Delta}_{f}^{-1}x^{-z}$ and
 \begin{equation}\label{(11*)}
 \begin{aligned}
 \hat{H}_{f}\Psi(f,z,x)&=\hat{\Delta}_{f}^{-1}(\hat{x}\hat{p}+\hat{p}\hat{x})\hat{\Delta}_{f}(\hat{\Delta}_{f}^{-1} x^{-z})\\
 &=\hat{\Delta}_{f}^{-1}(\hat{x}\hat{p}+\hat{p}\hat{x})x^{-z}\\
 &=\hat{\Delta}_{f}^{-1}[i(2z-1)x^{-z}]\\
 &=i(2z-1)\hat{\Delta}_{f}^{-1}x^{-z}\\
 &=i(2z-1)\Psi(f,z,x).
 \end{aligned}
 \end{equation}

Thus if we propose the boundary condition
$\Psi(f,z,0)=0$ to the differential equation (\ref{(11*)}), then we have
the $n$th eigenstates of the Hamiltonian $\hat{H}_{f}$ are
\begin{equation}\label{Eq1}
\Psi(f,z_{n},x)=(-1)^{z_{n}-1}L(f,z_{n},x+1),
\end{equation} 
the $n$th eigenvalues are $E_{n}=i(2z_{n}-1)$ for all $n.$
From the boundary condition $\Psi(f,z,0)=0,$ 
we have $\Psi(f,z_{n},0)=0$ and by (\ref{Eq1}) $L(f,z_{n},1)=0,$
thus $z_{n}=\frac{1}{2}(1-iE_{n})$ are the zeros of the general family zeta functions $L(f,z)$  
(since by (\ref{(6)}) $L(f,z,1)=L(f,z)$).
\end{proof}

\end{document}